\documentclass[12pt,reqno]{amsart}
\usepackage{amssymb}
\usepackage{eurosym}
\usepackage{amsfonts}
\usepackage{amsmath}
\usepackage[bookmarksnumbered, colorlinks, plainpages]{hyperref}

\hypersetup{colorlinks=true,linkcolor=red, anchorcolor=green, citecolor=cyan, urlcolor=red, filecolor=magenta, pdftoolbar=true}
\newtheorem{theorem}{Theorem}[section]
\newtheorem{lemma}[theorem]{Lemma}

\theoremstyle{definition}
\newtheorem{definition}[theorem]{Definition}

\theoremstyle{remark}
\newtheorem{remark}[theorem]{Remark}
\numberwithin{equation}{section}

\begin{document}
\date{{\scriptsize Received: , Accepted: .}}
\title[On the topological equivalence of $S$-metric and cone $S$-metric
spaces]{On the topological equivalence of $S$-metric and cone $S$-metric
spaces}
\subjclass[2010]{Primary 54H25; Secondary 47H10.}
\keywords{$S$-metric, cone $S$-metric, $N$-cone metric.}

\begin{abstract}
The aim of this paper is to establish the equivalence between the concepts of
an $S$-metric space and a cone $S$-metric space using\ some topological
approaches. We introduce a new notion of $TVS$-cone $S$-metric space using
some facts about topological vector spaces. We see that the known results on
cone $S$-metric spaces (or $N$-cone metric spaces) can be directly obtained
from the studies on $S$-metric spaces.
\end{abstract}

\author[N. TA\c{S}]{N\.{I}HAL TA\c{S}}
\address[Nihal Ta\c{s}]{Bal\i kesir University, Department of Mathematics,
10145 Bal\i kesir, TURKEY}
\email{nihaltas@balikesir.edu.tr}

\maketitle

\setcounter{page}{1}


\section{Introduction}

\label{sec:intro} The study of cone metric spaces was started with the paper
\cite{Guang-cone-metric}. Since then, various studies have been obtained on
cone metric spaces. But, using the topological aspects and some different
approaches, it was proved that the notions of a metric space and a\ cone
metric space are equivalent (for example, see \cite{Du-Tvs-cone}, \cite%
{Zafer-Ercan}, \cite{Kadelburg-2011} and \cite{Khani-2011} for more details).

Recently, $S$-metric spaces have been introduced as a generalization of
metric spaces \cite{Sedghi-S-metric}. The relationships between a metric and
an $S$-metric were given with some counter examples (see \cite{Gupta}, \cite%
{Hieu} and \cite{Ozgur-mathsci}). Then, Dhamodharan and Krishnakumar
introduced a new generalized metric space called as a cone $S$-metric space
\cite{Dhamodharan-cone-S-metric}. This metric space is also called as $N$%
-cone metric space by Malviya and Fisher in \cite{Malviya-N-Cone}. Some
well-known fixed-point results were generalized on both cone $S$-metric and $%
N$-cone metric spaces (for example, \cite{Dhamodharan-cone-S-metric}, \cite{Fernandez-2015} and \cite
{Malviya-N-Cone}).

In the present work, we show the topological equivalence of an $S$-metric
space and a cone $S$-metric space. To do this, we introduce a new notion
called as a $TVS$-cone $S$-metric space as a generalization of both metric
and cone $S$-metric (or $N$-cone metric) spaces. In Section \ref{sec:1}, we
recall some necessary definitions and lemmas in the sequel. In Section \ref%
{sec:2}, we present a notion of a $TVS$-cone $S$-metric space and establish
the equivalence between new this space and a cone $S$-metric space. Also, we
see that some known theorems studied on cone $S$-metric spaces (or $N$-cone
metric spaces) can be directly obtained from the studies on $S$-metric
spaces. In Section \ref{sec:3}, we investigate the relationships between an $%
S$-metric space and a cone $S$-metric space in view of their topological
properties. In Section \ref{sec:4}, we give a brief account of review about the obtained results and draw
a diagram which shows the relations among some known generalized metric
spaces.

\section{Preliminaries}

\label{sec:1} In this section, we recall some necessary notions and results
related to cone, $S$-metric and cone $S$-metric (or $N$-cone metric).

\begin{definition}
\cite{Sedghi-S-metric} \label{def1} Let $X$ be a nonempty set and $\mathcal{S%
}:X\times X\times X\rightarrow \left[ 0,\infty \right) $ be a function
satisfying the following conditions for all $x,y,z,a\in X:$

\begin{enumerate}
\item $\mathcal{S}(x,y,z)\geq 0$,

\item $\mathcal{S}(x,y,z)=0$ if and only if $x=y=z$,

\item $\mathcal{S}(x,y,z)\leq \mathcal{S}(x,x,a)+\mathcal{S}(y,y,a)+\mathcal{%
S}(z,z,a)$.
\end{enumerate}

Then the function $\mathcal{S}$ is called an $S$-metric on $X$ and the pair $%
(X,\mathcal{S})$ is called an $S$-metric space.
\end{definition}

\begin{definition}
\cite{Sedghi-S-metric} \label{def5} Let $(X,\mathcal{S})$ be an $S$-metric
space and $\left\{ x_{n}\right\} $ be a sequence in this space.

\begin{enumerate}
\item A sequence $\left\{ x_{n}\right\} \subset X$ converges to $x\in X$ if $%
\mathcal{S}(x_{n},x_{n},x)\rightarrow 0$ as $n\rightarrow \infty $, that is,
for each $\varepsilon >0$, there exists $n_{0}\in
\mathbb{N}
$ such that for all $n\geq n_{0}$ we have $\mathcal{S}(x_{n},x_{n},x)<%
\varepsilon $.

\item A sequence $\left\{ x_{n}\right\} \subset X$ is a Cauchy sequence if $%
\mathcal{S}(x_{n},x_{n},x_{m})\rightarrow 0$ as $n,m\rightarrow \infty $,
that is, for each $\varepsilon >0$, there exists $n_{0}\in
\mathbb{N}
$ such that for all $n,m\geq n_{0}$ we have $\mathcal{S}(x_{n},x_{n},x_{m})<%
\varepsilon $.

\item The $S$-metric space $(X,\mathcal{S})$ is complete if every Cauchy
sequence is a convergent sequence.
\end{enumerate}
\end{definition}

\begin{lemma}
\cite{Sedghi-S-metric} \label{lem3} Let $(X,\mathcal{S})$ be an $S$-metric
space and $x,y\in X$. Then we have%
\begin{equation*}
\mathcal{S}(x,x,y)=\mathcal{S}(y,y,x)\text{.}
\end{equation*}
\end{lemma}

\begin{definition}
\cite{Sedghi-S-metric} \label{def7} Let $(X,\mathcal{S})$ be an $S$-metric
space. For $r>0$ and $x\in X$, the open ball $B_{S}(x,r)$ defined as follows$%
:$%
\begin{equation*}
B_{S}(x,r)=\{y\in X:\mathcal{S}(y,y,x)<r\}.
\end{equation*}
\end{definition}

\begin{definition}
\cite{Guang-cone-metric} Let $E$ be a real Banach space and $P$ be a subset
of $E$. $P$ is called a cone if and only if
\end{definition}

\begin{enumerate}
\item $P$ is closed, nonempty and $P\neq \left\{ 0\right\} $,

\item If $a,b\in
\mathbb{R}
$, $a,b\geq 0$, $x,y\in P$ then $ax+by\in P$,

\item If $x\in P$ and $-x\in P$ then $x=0$.
\end{enumerate}

Given a cone $P\subset E$, a partial ordering $\precsim $ with respect to $P$
is defined by $x\precsim y$ if and only if $y-x\in P$. It was written $%
x\prec y$ to indicate that $x\precsim y$ but $x\neq y$. Also $x\ll y$ stands
for $y-x\in intP$ where $intP$ denotes the interior of $P$ \cite%
{Guang-cone-metric}.

\begin{lemma}
\cite{Khani-2011} \label{lem4} Let $(X,E)$ be a cone space with $x\in P$ and
$y\in intP$. Then one can find $n\in
\mathbb{N}
$ such that $x\ll ny$.
\end{lemma}

\begin{lemma}
\cite{Khani-2011} \label{lem5} Let $y\in intP$. If $x\geq y$ for all $x$
then $x\in intP$.
\end{lemma}

\begin{lemma}
\cite{Khani-2011} \label{lem6} Let $(E,P)$ be a cone space. If $x\leq y\ll z$
then $x\ll z$.
\end{lemma}

\begin{definition}
\cite{Dhamodharan-cone-S-metric} \label{def2} Suppose that $E$ is a real
Banach space, then $P$ is a cone in $E$ with $intP\neq \emptyset $ and $%
\precsim $ is partial ordering with respect to $P$. Let $X$ be a nonempty
set, a function $\mathcal{S}:X\times X\times X\rightarrow E$ is called a
cone $S$-metric on $X$ if it satisfies the following conditions

\begin{enumerate}
\item $0\precsim \mathcal{S}(x,y,z)$,

\item $\mathcal{S}(x,y,z)=0$ if and only if $x=y=z$,

\item $\mathcal{S}(x,y,z)\precsim \mathcal{S}(x,x,a)+\mathcal{S}(y,y,a)+%
\mathcal{S}(z,z,a)$.
\end{enumerate}

Then the function $\mathcal{S}$ is called a cone $S$-metric on $X$ and the
pair $(X,\mathcal{S})$ is called a cone $S$-metric space.
\end{definition}

We note that the notion of a cone $S$-metric is also called as an $N$-cone
metric in \cite{Malviya-N-Cone}. We use also some results from \cite%
{Malviya-N-Cone}.

\begin{lemma}
\cite{Dhamodharan-cone-S-metric} \label{lem1} Let $(X,\mathcal{S})$ be a
cone $S$-metric space. Then we get%
\begin{equation*}
\mathcal{S}(x,x,y)=\mathcal{S}(y,y,x)\text{.}
\end{equation*}
\end{lemma}

\begin{definition}
$($Topology of $N$-cone metric space$)$ \cite{Fernandez-2015} \label{def6}
Let $(X,N)$ be an $N$-cone metric space, each $N$-cone metric $N$ on $X$
generates a topology $\tau _{N}$ on $X$ whose base is the family of open $N$%
-balls defined as%
\begin{equation*}
B_{N}(x,c)=\left\{ y\in X:N(y,y,x)\ll c\right\} \text{,}
\end{equation*}%
for $c\in E$ with $0\ll c$ and for all $x\in X$.
\end{definition}

Definition \ref{def6} can also be considered for cone $S$-metric spaces as
follows:

Let $(X,\mathcal{S})$ be a cone $S$-metric space. The open ball $B_{S}(x,c)$
is defined by%
\begin{equation*}
B_{S}(x,c)=\{y\in X:\mathcal{S}(y,y,x)\ll c\},
\end{equation*}%
for $c\in E$ with $0\ll c$ and for all $x\in X$.

\section{$TVS$-cone $S$-metric spaces}

\label{sec:2} Let $E$ be a topological vector space (briefly $TVS$) with its
zero vector $\theta _{E}$. A nonempty subset $K$ of $E$ is called a convex
cone if $K+K\subseteq K$ and $\lambda K\subseteq K$ for $\lambda \geq 0$. A
convex cone $K$ is said to be pointed if $K\cap (-K)=\left\{ \theta
_{E}\right\} $. For a given cone $K\subseteq E$, a partial ordering $%
\precsim _{K}$ with respect to $K$ is defined by%
\begin{equation*}
x\precsim _{K}y\Longleftrightarrow y-x\in K\text{.}
\end{equation*}%
$x\prec _{K}y$ stands for $x\precsim _{K}y$ and $x\neq y$. Also $x\ll y$
stands for $y-x\in intK$ where $intK$ denotes the interior of $K$ \cite%
{Du-Tvs-cone}.

Let $Y$ be a locally convex Hausdorff $TVS$ with its zero vector $\theta $, $%
K$ be a proper, closed and convex pointed cone in $Y$ with $intK\neq
\emptyset $, $e\in intK$ and $\precsim _{K}$ be a partial ordering with
respect to $K$. The nonlinear scalarization function $\xi _{e}:Y\rightarrow
\mathbb{R}
$ is defined by%
\begin{equation*}
\xi _{e}(y)=\inf \left\{ r\in
\mathbb{R}
:y\in re-K\right\} \text{,}
\end{equation*}%
for all $y\in Y$ (see \cite{Chen-2005}, \cite{Du-2008}, \cite{Gerth-1990},
\cite{Gopfert-2000} and \cite{Gopfert-2003} for more details).

We recall the following lemma given in \cite{Chen-2005}, \cite{Du-2008},
\cite{Gerth-1990}, \cite{Gopfert-2000} and \cite{Gopfert-2003}.

\begin{lemma}
\label{lem2} For each $r\in
\mathbb{R}
$ and $y\in Y$, the following statements are satisfied$:$

\begin{enumerate}
\item $\xi _{e}(y)\leq r$ if and only if $y\in re-K$,

\item $\xi _{e}(y)>r$ if and only if $y\notin re-K$,

\item $\xi _{e}(y)\geq r$ if and only if $y\notin re-intK$,

\item $\xi _{e}(y)<r$ if and only if $y\in re-intK$,

\item $\xi _{e}(.)$ is positively homogeneous and continuous on $Y$,

\item If $y_{1}\in y_{2}+K$ then $\xi _{e}(y_{2})\leq \xi _{e}(y_{1})$,

\item $\xi _{e}(y_{1}+y_{2})\leq \xi _{e}(y_{1})+\xi _{e}(y_{2})$ for all $%
y_{1},y_{2}\in Y$.
\end{enumerate}
\end{lemma}

Now we introduce the notion of a $TVS$-cone $S$-metric space.

\begin{definition}
\label{def3} Let $X$ be a nonempty set and $\mathcal{S}_{p}:X\times X\times
X\rightarrow Y$ be a vector-valued function. If the following conditions hold%
$:$

\begin{enumerate}
\item $\theta \precsim _{K}\mathcal{S}_{p}(x,y,z)$,

\item $\mathcal{S}_{p}(x,y,z)=\theta $ if and only if $x=y=z$,

\item $\mathcal{S}_{p}(x,y,z)\precsim _{K}\mathcal{S}_{p}(x,x,a)+\mathcal{S}%
_{p}(y,y,a)+\mathcal{S}_{p}(z,z,a)$,
\end{enumerate}

for all $x,y,z,a\in X$, then the function $\mathcal{S}_{p}$ is called a $TVS$%
-cone $S$-metric and the pair $(X,\mathcal{S}_{p})$ is called a $TVS$-cone $%
S $-metric space.
\end{definition}

\begin{remark}
\label{rem1} A cone $S$-metric space is a special case of a $TVS$-cone $S$%
-metric space.
\end{remark}

\begin{theorem}
\label{thm1} Let $(X,\mathcal{S}_{p})$ be a $TVS$-cone $S$-metric space.
Then the function $\mathcal{S}_{p}^{S}:X\times X\times X\rightarrow \left[
0,\infty \right) $ defined by%
\begin{equation*}
\mathcal{S}_{p}^{S}=\xi _{e}\circ \mathcal{S}_{p}\text{,}
\end{equation*}%
is an $S$-metric.
\end{theorem}

\begin{proof}
Using the condition (1) given in Definition \ref{def3} and Lemma \ref{lem2},
we get $\mathcal{S}_{p}^{S}(x,y,z)\geq 0$ for all $x,y,z\in X$. From the
condition (2) given in Definition \ref{def3} and Lemma \ref{lem2}, we obtain
the following cases:

\textbf{Case 1:} If $x=y=z$, then we have
\begin{equation*}
\mathcal{S}_{p}^{S}(x,y,z)=\xi _{e}\circ \mathcal{S}_{p}(x,y,z)=\xi
_{e}(\theta )=0\text{.}
\end{equation*}

\textbf{Case 2:} If $\mathcal{S}_{p}^{S}(x,y,z)=0$, then we have%
\begin{equation*}
\xi _{e}\circ \mathcal{S}_{p}(x,y,z)=0\Rightarrow \mathcal{S}_{p}(x,y,z)\in
K\cap (-K)=\left\{ \theta \right\} \Rightarrow x=y=z\text{.}
\end{equation*}%
If we apply the condition (3) given in Definition \ref{def3} together with
the conditions (6) and (7) given in Lemma \ref{lem2}, then we obtain%
\begin{eqnarray*}
\mathcal{S}_{p}^{S}(x,y,z) &=&\xi _{e}\circ \mathcal{S}_{p}(x,y,z) \\
&\leq &\xi _{e}\left( \mathcal{S}_{p}(x,x,a)+\mathcal{S}_{p}(y,y,a)+\mathcal{%
S}_{p}(z,z,a)\right) \\
&\leq &\xi _{e}\left( \mathcal{S}_{p}(x,x,a)+\mathcal{S}_{p}(y,y,a)\right)
+\xi _{e}\left( \mathcal{S}_{p}(z,z,a)\right) \\
&\leq &\xi _{e}\left( \mathcal{S}_{p}(x,x,a)\right) +\xi _{e}\left( \mathcal{%
S}_{p}(y,y,a)\right) +\xi _{e}\left( \mathcal{S}_{p}(z,z,a)\right) \\
&=&\mathcal{S}_{p}^{S}(x,x,a)+\mathcal{S}_{p}^{S}(y,y,a)+\mathcal{S}%
_{p}^{S}(z,z,a)\text{,}
\end{eqnarray*}%
for all $x,y,z,a\in X$. Therefore, $\mathcal{S}_{p}^{S}$ is an $S$-metric.
\end{proof}

\begin{remark}
\label{rem2} Let $(X,\mathcal{S})$ be a cone $S$-metric space. Then the
function $\mathcal{S}_{p}^{S}:X\times X\times X\rightarrow \left[ 0,\infty
\right) $ defined by
\begin{equation*}
\mathcal{S}_{p}^{S}=\xi _{e}\circ \mathcal{S}\text{,}
\end{equation*}%
is an $S$-metric.
\end{remark}

Using the ideas of \cite{Dhamodharan-cone-S-metric} and \cite{Malviya-N-Cone}%
, we give the following definition.

\begin{definition}
\label{def4} Let $(X,\mathcal{S}_{p})$ be a $TVS$-cone $S$-metric space, $%
x\in X$ and $\left\{ x_{n}\right\} $ be a sequence in $X$.

\begin{enumerate}
\item $\left\{ x_{n}\right\} $ converges to $x$ if and only if $\mathcal{S}%
_{p}(x_{n},x_{n},x)\rightarrow \theta $ as $n\rightarrow \infty $, that is,
there exists $n_{0}\in
\mathbb{N}
$ such that for all $n\geq n_{0}$, $\mathcal{S}_{p}(x_{n},x_{n},x)\ll c$ for
each $c\in Y$ with $\theta \ll c$. It is denoted by%
\begin{equation*}
\underset{n\rightarrow \infty }{\lim }x_{n}=x\text{.}
\end{equation*}

\item $\left\{ x_{n}\right\} $ is a Cauchy sequence if $\mathcal{S}%
_{p}(x_{n},x_{n},x_{m})\rightarrow \theta $ as $n,m\rightarrow \infty $,
that is, there exists $n_{0}\in
\mathbb{N}
$ such that for all $n,m\geq n_{0}$, $\mathcal{S}_{p}(x_{n},x_{n},x_{m})\ll
c $ for each $c\in Y$ with $\theta \ll c$.

\item $(X,\mathcal{S}_{p})$ is complete if every Cauchy sequence in $X$ is
convergent.
\end{enumerate}
\end{definition}

\begin{theorem}
\label{thm2} Let $(X,\mathcal{S}_{p})$ be a $TVS$-cone $S$-metric space, $%
x\in X$, $\left\{ x_{n}\right\} $ be a sequence in $X$ and $\mathcal{S}%
_{p}^{S}$ be defined as in Theorem \ref{thm1}. Then the following statements
hold$:$

\begin{enumerate}
\item If $\left\{ x_{n}\right\} $ converges to $x$ in $(X,\mathcal{S}_{p})$,
then $\left\{ x_{n}\right\} $ converges to $x$ in $(X,\mathcal{S}_{p}^{S})$.

\item If $\left\{ x_{n}\right\} $ is a Cauchy sequence in $(X,\mathcal{S}%
_{p})$, then $\left\{ x_{n}\right\} $ is a Cauchy sequence in $(X,\mathcal{S}%
_{p}^{S})$.

\item If $(X,\mathcal{S}_{p})$ is complete, then $(X,\mathcal{S}_{p}^{S})$
is complete.
\end{enumerate}
\end{theorem}

\begin{proof}
(1) Let $\varepsilon >0$ be given. Using Lemma \ref{lem2} and Theorem \ref%
{thm1}, if $\left\{ x_{n}\right\} $ converges to $x$ in $(X,\mathcal{S}_{p})$%
, then there exists $n_{0}\in
\mathbb{N}
$ such that%
\begin{equation*}
\mathcal{S}_{p}(x_{n},x_{n},x)\ll \varepsilon e\Longleftrightarrow \mathcal{S%
}_{p}^{S}(x_{n},x_{n},x)=\xi _{e}\circ \mathcal{S}_{p}(x_{n},x_{n},x)<%
\varepsilon \text{,}
\end{equation*}%
for all $n\geq n_{0}$ since $e\in intK$. Therefore, the condition (1) hold.

(2) Let $\left\{ x_{n}\right\} $ be a Cauchy sequence in $(X,\mathcal{S}%
_{p}) $. Then there exists $n_{0}\in
\mathbb{N}
$ such that%
\begin{equation*}
\mathcal{S}_{p}(x_{n},x_{n},x_{m})\ll \varepsilon e\Longleftrightarrow
\mathcal{S}_{p}^{S}(x_{n},x_{n},x_{m})<\varepsilon \text{,}
\end{equation*}%
for all $n,m\geq n_{0}$. Hence, $\left\{ x_{n}\right\} $ is a Cauchy
sequence in $(X,\mathcal{S}_{p}^{S})$.

(3) From the conditions (1) and (2), the condition (3) holds.
\end{proof}

\begin{theorem}
\label{thm3} Let $(X,\mathcal{S}_{p})$ be a complete $TVS$-cone $S$-metric
space and the self-mapping $T:X\rightarrow X$ satisfies the condition%
\begin{equation*}
\mathcal{S}_{p}(Tx,Tx,Ty)\precsim _{K}h\mathcal{S}_{p}(x,x,y)\text{,}
\end{equation*}%
for all $x,y\in X$ and some $h\in \left[ 0,1\right) $. Then $T$ has a unique
fixed point in $X$.
\end{theorem}

\begin{proof}
Using Theorem \ref{thm1} and Theorem \ref{thm2}, we obtain that $(X,\mathcal{%
S}_{p}^{S})$ is a complete $S$-metric space. From Lemma \ref{lem2}, we get%
\begin{equation*}
\mathcal{S}_{p}(Tx,Tx,Ty)\precsim _{K}h\mathcal{S}_{p}(x,x,y)\Longrightarrow
\mathcal{S}_{p}^{S}(Tx,Tx,Ty)\leq h\mathcal{S}_{p}^{S}(x,x,y)\text{,}
\end{equation*}%
for all $x,y\in X$. Then the proof is easily seen from Theorem 3.1 on page
263 in \cite{Sedghi-S-metric}.
\end{proof}

\begin{remark}
\label{rem3} $(1)$ Theorem \ref{thm3}, Theorem $3.1$ $($on page 263 in \cite%
{Sedghi-S-metric}$)$ and Theorem $2.1$ $($on page 239 in \cite%
{Dhamodharan-cone-S-metric}$)$ are equivalent.

$(2)$ By the similar arguments used in the proof of Theorem \ref{thm3}, we
obtain the following relations$:$

$(i)$ Theorem $2.5$ $($on page 242 in \cite{Dhamodharan-cone-S-metric}$)$
and Theorem $4$ $($on page 244 in \cite{Ozgur-chapter}$)$ are equivalent.

$(ii)$ Theorem $2.3$ $($on page 240 in \cite{Dhamodharan-cone-S-metric}$)$
and Theorem $3$ $($on page 240 in \cite{Ozgur-chapter}$)$ are equivalent.

$(iii)$ Theorem $2.1$ $($on page 7 in \cite{Malviya-N-Cone}$)$ and Corollary
$2.19$ $($on page 122 in \cite{Sedghi-2014}$)$ are equivalent.

$(iv)$ Theorem $2.1$ $($on page 35 in \cite{Fernandez-2015}$)$ and Theorem $%
3.1$ $($on page 263 in \cite{Sedghi-S-metric}$)$ are equivalent.

$(v)$ Theorem $2.2$ $($on page 35 in \cite{Fernandez-2015}$)$ and Corollary $%
2.8$ $($on page 118 in \cite{Sedghi-2014}$)$ are equivalent.

$(vi)$ Theorem $2.3$ $($on page 36 in \cite{Fernandez-2015}$)$ and Corollary
$2.15$ $($on page 121 in \cite{Sedghi-2014}$)$ are equivalent.
\end{remark}

\section{Topological equivalence of $S$-metric and cone $S$-metric spaces}

\label{sec:3}

In the following theorem, we give the topological equivalence of an $S$%
-metric and a cone $S$-metric space.

\begin{theorem}
\label{thm4} Let $(X,\mathcal{S})$ be a cone $S$-metric space $($or $N$-cone
metric space$)$, $e\in intP$ and $a\in \left[ 0,\infty \right) $ with $a<1$.
Then there exists an $S$-metric $\mathcal{S}^{\ast }:X\times X\times
X\rightarrow \left[ 0,\infty \right) $ which induces the same topology on $X$
as the cone $S$-metric topology induced by $\mathcal{S}$.
\end{theorem}

\begin{proof}
At first, we put $h=\frac{1}{a}$ and define a function $\Theta :X\times
X\times X\rightarrow \left[ 0,\infty \right) $ as%
\begin{equation}
\Theta (x,y,z)=\left\{
\begin{array}{ccc}
h^{\min \left\{ \alpha :\mathcal{S}(x,y,z)\ll h^{\alpha }e\right\} } & \text{%
if} & \mathcal{S}(x,y,z)\neq 0 \\
0 & \text{if} & \mathcal{S}(x,y,z)=0%
\end{array}%
\right. ,  \label{eqn0}
\end{equation}%
where $\alpha \in
\mathbb{Z}
$. It can be easily checked that%
\begin{equation*}
\Theta (x,x,y)=\Theta (y,y,x)
\end{equation*}%
and%
\begin{equation*}
\Theta (x,y,z)=0\Longleftrightarrow x=y=z\text{.}
\end{equation*}%
Now we define the function $\mathcal{S}^{\ast }:X\times X\times X\rightarrow %
\left[ 0,\infty \right) $ by%
\begin{equation}
\mathcal{S}^{\ast }(x,y,z)=\inf \left\{ \sum\limits_{i=1}^{n-2}\Theta
(x_{i},x_{i+1},x_{i+2}):x_{1}=x,\ldots ,x_{n-1}=y,x_{n}=z\right\} \text{.}
\label{eqn1}
\end{equation}%
From the definitions (\ref{eqn0}) and (\ref{eqn1}), we have $\mathcal{S}%
^{\ast }(x,y,z)\geq 0$ and%
\begin{equation*}
\mathcal{S}^{\ast }(x,y,z)=0\Longleftrightarrow x=y=z\text{.}
\end{equation*}%
We show that the triangle inequality is satisfied by the function $\mathcal{S%
}^{\ast }$. For $\varepsilon >0$, we prove%
\begin{equation*}
\mathcal{S}^{\ast }(x,y,z)\leq \mathcal{S}^{\ast }(x,x,a)+\mathcal{S}^{\ast
}(y,y,a)+\mathcal{S}^{\ast }(z,z,a)+\varepsilon \text{.}
\end{equation*}%
By the definition (\ref{eqn1}), there exists $x_{1}=x,\ldots
,x_{n-1}=x,x_{n}=a$ with%
\begin{equation*}
\sum \Theta (x_{i},x_{i},x_{i+1})\leq \mathcal{S}^{\ast }(x,x,a)+\frac{%
\varepsilon }{3}\text{,}
\end{equation*}%
$y_{1}=y,\ldots ,y_{n-1}=y,y_{n}=a$ with%
\begin{equation*}
\sum \Theta (y_{i},y_{i},y_{i+1})\leq \mathcal{S}^{\ast }(y,y,a)+\frac{%
\varepsilon }{3}
\end{equation*}%
and $z_{1}=z,\ldots ,z_{n-1}=z,z_{n}=a$ with%
\begin{equation*}
\sum \Theta (z_{i},z_{i},z_{i+1})\leq \mathcal{S}^{\ast }(z,z,a)+\frac{%
\varepsilon }{3}\text{.}
\end{equation*}%
Therefore, we get%
\begin{eqnarray*}
\mathcal{S}^{\ast }(x,y,z) &\leq &\sum \Theta (x_{i},x_{i},x_{i+1})+\sum
\Theta (y_{i},y_{i},y_{i+1})+\sum \Theta (z_{i},z_{i},z_{i+1}) \\
&\leq &\mathcal{S}^{\ast }(x,x,a)+\mathcal{S}^{\ast }(y,y,a)+\mathcal{S}%
^{\ast }(z,z,a)+\varepsilon \text{,}
\end{eqnarray*}%
that is, $\mathcal{S}^{\ast }$ is an $S$-metric.

Now we show that each $B_{S}(x,c)$ contains some $B_{S^{\ast }}(x,r)$. Let
us consider the open ball $B_{S^{\ast }}(x,r)$ for $x\in X$ and $r\in \left[
0,\infty \right) $. It can be found $\alpha \in
\mathbb{Z}
$ such that $h^{\alpha }<r$. We put $c\ll h^{\alpha }e$. If $\mathcal{S}%
(x,x,y)\ll c$ then $\Theta (x,x,y)\leq h^{\alpha }<r$ and $\mathcal{S}^{\ast
}(x,x,y)\leq \Theta (x,x,y)<r$, for each $y\in X$. Then we get%
\begin{equation}
B_{S}(x,c)\subseteq B_{S^{\ast }}(x,r)\text{.}  \label{eqn2}
\end{equation}%
Conversely, let us consider the open ball $B_{S}(x,c)$ for $x\in X$ and $%
c\in E$. For each $x,y\in X$ and $r\in \left[ 0,\infty \right) $ if $%
\mathcal{S}^{\ast }(x,x,y)<r$ then we can find $x_{1}=x,\ldots
,x_{n-1}=x,x_{n}=y$ with%
\begin{equation*}
\sum \Theta (x_{i},x_{i},x_{i+1})<r\text{.}
\end{equation*}%
However for each $i<n$, we have%
\begin{equation*}
\mathcal{S}(x_{i},x_{i},x_{i+1})\ll \Theta (x_{i},x_{i},x_{i+1})e
\end{equation*}%
and so%
\begin{equation*}
\mathcal{S}(x,x,y)\leq \sum_{i=1}^{n-1}\Theta (x_{i},x_{i},x_{i+1})e\leq re%
\text{.}
\end{equation*}%
If we choose $r$ satisfying $re\ll c$, then we have%
\begin{equation*}
\mathcal{S}(x,x,y)\ll c
\end{equation*}%
and%
\begin{equation}
B_{S^{\ast }}(x,r)\subseteq B_{S}(x,c)\text{.}  \label{eqn3}
\end{equation}%
Therefore, from the inequalities (\ref{eqn2}) and (\ref{eqn3}), $\mathcal{S}%
^{\ast }$ induces the same topology as the cone $S$-metric topology of $S$.
\end{proof}

\section{Conclusion}

\label{sec:4}

We have defined the concept of a $TVS$-cone $S$-metric space as a
generalization of a cone $S$-metric space. We have established the
equivalence between the notions of an $S$-metric space and a $TVS$-cone $S$%
-metric space (resp. cone $S$-metric space) and presented some related
results. Also it is shown the topological equivalence of these spaces. On
the other hand, complex valued $S$-metric spaces are a special class of cone
$S$-metric spaces. But it is important to study some fixed-point results in
complex valued $S$-metric spaces since some contractions have a product and
quotient (see \cite{Mlaiki-complex-S-metric} and \cite{Tas-complex} for more
details).

From the known (see \cite{Dhamodharan-cone-S-metric}, \cite{Du-Tvs-cone},
\cite{Zafer-Ercan}, \cite{Guang-cone-metric}, \cite{Gupta}, \cite{Hieu},
\cite{Kadelburg-2011}, \cite{Khani-2011}, \cite{Malviya-N-Cone} and \cite%
{Ozgur-mathsci} for more details) and obtained results, we get the following
diagram:\newline

\ \ \ \ \ \ metric spaces $\Longleftrightarrow $ cone metric spaces

\ \ \ \ \ \ \ \ \ \ \ \ \ \ $\Downarrow $ \ \ \ \ \ \ \ \ \ \ \ \ \ \ \ \ \
\ \ \ \ \ \ \ \ \ \ \ \ $\Downarrow $

$\ \ \ S$-metric spaces $\Longleftrightarrow $ cone $S$-metric spaces = $N$%
-cone metric spaces

\end{document}